\documentclass[11pt,reqno]{amsart}
\usepackage{amsmath,amssymb,amsfonts,amsthm,amscd,amstext,amsxtra,amsopn,array,url,verbatim,mathrsfs}
\usepackage{graphicx}
\usepackage{amsmath,amssymb,amsfonts,amsthm,amssymb,amscd,amstext,amsxtra,amsopn,array,url,verbatim,mathrsfs}

\usepackage{url}
\usepackage{booktabs}

\usepackage{longtable}

\usepackage{fullpage}
\usepackage{times}
\usepackage{dsfont}

\makeatletter
\@namedef{subjclassname@2010}{%
  \textup{2010} Mathematics Subject Classification}
\makeatother

\newtheorem{theorem}{Theorem}[section]
\newtheorem{lemma}[theorem]{Lemma}
\newtheorem{prop}[theorem]{Proposition}
\newtheorem{corollary}[theorem]{Corollary}
\numberwithin{equation}{section}

\theoremstyle{remark}

\renewcommand{\mod}[1]{{\ifmmode\text{\rm\ (mod~$#1$)}\else\discretionary{}{}{\hbox{ }}\rm(mod~$#1$)\fi}}

\newcommand{\N}{{\mathbb{N}}}
\renewcommand{\Re}{{\mathfrak{Re}}}

\renewcommand{\b}{\beta}
\newcommand{\g}{\gamma}

\renewcommand{\t}{\theta}
\newcommand{\s}{\sigma}
\renewcommand{\d}{\delta}

\begin{document}

\title{Short effective intervals containing primes.}
\author[Habiba Kadiri]{Habiba Kadiri}
\address{Department of Mathematics and Computer Science, University of Lethbridge, 4401 University Drive, Lethbridge, Alberta, T1K 3M4 Canada}
\email{habiba.kadiri@uleth.ca}
\author[Allysa Lumley]{Allysa Lumley}
\address{Department of Mathematics and Computer Science, University of Lethbridge, 4401 University Drive, Lethbridge, Alberta, T1K 3M4 Canada}
\email{allysa.lumley@uleth.ca}
\thanks{Our calculations were done on the University of Lethbridge Number Theory Group Eudoxus machine, supported by an NSERC RTI grant.}

\begin{abstract}
We prove that if $x$ is large enough, namely $x\ge x_0$, then there exists a prime between $x(1- \Delta^{-1})$ and $x$, where $\Delta$ is an effective constant computed in terms of $x_0$.
\end{abstract}

\maketitle
\section{Introduction.}\label{intro}
In this article, we address the problem of finding short intervals containing primes.
In $1845$ Bertrand conjectured that for any integer $n > 3$, there always exists at least one prime number $p$ with $n < p < 2n - 2$.
This was proven by Chebyshev in $1850$, using elementary methods.
Since then other intervals of the form $(kn,(k+1)n)$ have been investigated.
We refer the reader to \cite{Ba} for $k=2$, and to \cite{Loo} for $k=3$. 
Assuming that $x$ is arbitrarily large, the length of intervals containing primes can be drastically reduced.
To date, the record is held by Baker, Harman, and Pintz \cite{BHP} as they prove that there is at least one prime between $x$ and $x+x^{0.525+\varepsilon}$. 
This is an impressive result since under the Riemann Hypothesis the exponent $0.525$ can only be reduced to $0.5$. 
On the other hand, maximal gaps for the first primes have been checked numerically up to $4\cdot 10^{18}$ by Oliveira e Silva et al. \cite{SilHerPar}. 
In particular, they find that the largest prime gap before this limit is $1\,476$ and occurs at $ 1\,425\,172\,824\,437\,699\,411 =  e^{41.8008\ldots}$.
The purpose of this article is to obtain an effective result of the form:
for all $x\ge x_0$, there exists $\Delta>0$ such that the interval $(x(1-\Delta^{-1}),x)$ contains at least one prime.
In $1976$ Schoenfeld's \cite[Theorem 12]{Sh} gave this for $x_0= 2\,010\,881.1 $ and $\Delta=16\,598$.
In $2003$ Ramar\'e and Saouter improved on Schoenfeld's method by using a smoothing argument.
They also extended the computations to many other values for $x_0$ (\cite[Theorem 2 and Table 1]{RaSa}).
In \cite{Kad3}, the first author generalized this theorem to primes in arithmetic progression and applied this to Waring's seven cubes problem.
Here, our theorem improves \cite{RaSa} by making use of a new explicit zero-density for the zeros of the Riemann zeta function:
\begin{theorem} \label{main-thm} 
Let $x_0\ge 4\cdot 10^{18}$ be a fixed constant and let $x> x_0$.
Then there exists at least one prime $p$ such that $ (1-\Delta^{-1})x < p < x$, where $\Delta$ is  a constant depending on $x_0$ and is given in Table \ref{Table}.  
\end{theorem}
In Section \ref{proof}, we prove a general theorem (Theorem \ref{1st-main-thm}) which provides conditions for intervals of the form $((1-\Delta^{-1})x , x)$ to contain a prime.
In Section \ref{Computations}, we apply this theorem to compute explicit values for $\Delta$.

We present an example of numerical improvement this theorem allows, for instance when $x_0=e^{59}$.
Ramar\'e and Saouter \cite{RaSa} found that the interval gap was given by $\Delta = 209\,257\,759$.
In \cite[page 74]{HH1}, Helfgott mentioned an improvement of Ramar\'e using Platt's latest verification of the Riemann Hypothesis \cite{Pla0}: $\Delta =307\,779\,681$.
Our Theorem \ref{main-thm} leads to $\Delta =1~946~282~821$.

We now mention an application to the verification of the Ternary Goldbach conjecture.
This conjecture was known to be true for sufficiently large integers (by Vinogradov), and  Liu and Wang \cite{LiuWang} prove it for all integers $n \ge e^{3100}$.
On the other hand, the conjecture was verified for the first values of $n$. 
In \cite[Corollary 1]{RaSa}, Ramar\'e and Saouter verified it for $n\le 1.132\cdot 10^{22}$. 
Very recently, Oliveira e Silva et. al. \cite[Theorem 2.1]{SilHerPar} extended this limit to $ n\le 8.370\cdot 10^{26}$.
In \cite[Proposition A.1.]{HH1}, Helfgott applied the above result on short intervals containing primes ($\Delta=307\,779\,681$) and found $n \le 1.231\cdot10^{27}$. 
This allowed him to complete his proof \cite{HH1} \cite{HH2} of the Ternary Goldbach conjecture for the remaining integers.
Here our main theorem gives:
\begin{corollary}\label{Goldbach}
Every odd number larger than $5$ and smaller than 
\[1~966~196~911\times 4\cdot10^{18}= 7.864\ldots \cdot10^{27}\]
is the sum of at most three primes. 
\end{corollary}
As of today, Helfgott and Platt \cite{HP} have announced a verification up to $8.875\cdot10^{30}$.

\section{Proof of Theorem \ref{main-thm}} \label{proof}
We recall the definition of the classical Chebyshev functions:
\[
\theta(x) = \sum_{p\le x} \log p,\quad
\psi(x) = \sum_{n\le x} \Lambda (n),\ 
\text{with}\ \Lambda(n) = 
\begin{cases}
1 & \text{ if } n = p^k \text{ for some } k\in\N, \\
0 & \text{ otherwise}.
\end{cases}
\]
For each $x_0$, we want to find the largest $\Delta>0$ such that, for all $x>x_0$, there exists a prime between $x(1-\Delta^{-1})$ and $x$. 
This happens as soon as 
\[
\theta(x) - \theta(x(1-\Delta^{-1})) > 0.
\]
\subsection{Introduction of parameters} 
We list here the parameters we will be using throughout the proof.
\begin{equation}
\begin{split}
\label{cond}
*\ & m \text{ integer with } m\ge2,\\ 
*\ & 0\le u \le 0.0001,\ 
 \d = mu \ \text{and}\  0\le \d \le 0.0001 ,\\
*\ & 0\le a\le 1/2 ,\\
*\ & \Delta = \left( 1- (1+\d a) (1+\d(1-a))^{-1} e^{-u}\right)^{-1},\\
*\ & X \ge X_0\ge e^{38},\\
*\ & x=e^uX(1+\d(1-a)) \ge x_0 = e^uX_0(1+\d(1-a)),\\
*\ & y = X(1+\d a) =x\left(1-\Delta^{-1}\right) .
\end{split}
\end{equation}
\subsection{Smoothing the difference $\theta(x)-\theta(y)$} 
We follow here the smoothing argument of \cite{RaSa}.
Let $f$ be a positive function integrable on $(0,1)$. We denote
\begin{align}
& \label{def-norm1f}
 \|f\|_1 = \int_0^1 f(t) dt , \\
& \label{def-nufa}
 \nu(f,a) = \int_0^a f(t) dt +\int_{1-a}^1 f(t) dt,
\\& 
\label{def-I}
\text{and}\ 
I_{\d,u,X} = \frac1{\|f\|_1}\int_{0}^{1} \left(  \theta(e^uX(1+\d t) )-\theta(X(1+\d t) )\right) f(t) dt  .
\end{align}
Note that for all $a \le t \le 1-a$, $ \theta(e^uX(1+\d t) )-\theta(X(1+\d t) ) \le \theta(x)-\theta(y)$. 
We integrate with the positive weight $f$ and obtain:
\begin{equation}\label{ineq-2}
\int_{a}^{1-a} \left(  \theta(e^uX(1+\d t) )-\theta(X(1+\d t) )\right) f(t) dt \le \left(\theta(x)-\theta(y)\right) \int_{a}^{1-a} f(t) dt.
\end{equation}
We extend the left integral to the interval $(0,1)$ and use a Brun-Titchmarsh inequality to control the primes on the extremities $(0,a)$ and $(1-a,1)$ of the interval (see \cite[page 16, line -5]{RaSa} or \cite[Theorem 2]{MV}):
\begin{multline}\label{ineq-3}
\int_{t\in (0,a)\cup(1-a,1)} \left(  \theta(e^uX(1+\d t) )-\theta(X(1+\d t) )\right) f(t) dt 
\\ \le 2(1+\d )(e^u-1)\frac{\log(e^uX)}{\log(X(e^u-1))} \nu(f,a) X.
\end{multline}
Note that \cite{RaSa} uses the slightly larger bound
\[ 
2.0004 u \frac{\log X}{\log(uX)} \nu(f,a)  X .
\]
Combining \eqref{ineq-2} and \eqref{ineq-3} gives for $I_{\d,u,X}$:
\begin{equation}\label{ineq-23}
I_{\d,u,X} \le \left(\theta(x)-\theta(y)\right) \frac{\int_{a}^{1-a} f(t) dt}{\|f\|_1} + 2(1+\d )(e^u-1)\frac{\log(e^uX(1+\d ))}{\log(X(e^u-1))} \frac{ \nu(f,a) }{\|f\|_1}   X.
\end{equation}
Thus $\theta(x)-\theta(y)>0$ when
\begin{equation}\label{ineq-4}
I_{\d,u,X} - 2(1+\d )(e^u-1)\frac{\log(e^uX(1+\d ))}{\log(X(e^u-1))} \frac{ \nu(f,a) }{\|f\|_1} X>0.
\end{equation}
It remains to establish a lower bound for $I_{\d,u,X}$.
To do so, we first approximate $\theta(x)$ with $\psi(x)$.
This will allow us to translate our problem in terms of the zeros of the zeta function.
We use approximations proven by Costa in \cite[Theorem 5]{Costa}:
\begin{lemma}
\label{ineq-psi-theta} 
Let $x\ge e^{38}$. Then
\begin{equation}
 0.999\sqrt{x}+\sqrt[3]{x}<\psi(x)-\theta(x) <1.001\sqrt{x}+\sqrt[3]{x}.
\end{equation}
\end{lemma}
Then we have that for all $0<t<1$,
\begin{multline}\label{diff-psi-theta}
\left( \psi(e^u X (1+\d t)) - \theta(e^uX(1+\d t)) \right)
- \left( \psi( X(1+\d t))-\theta( X(1+\d t)) \right)
\\ 
<
\sqrt{X} \sqrt{1+\d}\left(
1.001e^{u/2}-0.999+X^{-1/6}(1+\d)^{-1/6}(e^{u/3}-1)
\right)
<\omega \sqrt{X},
\end{multline}
where we can take, under our assumptions \eqref{cond},
\begin{equation}\label{def-omega}
\omega= 2.05022\cdot10^{-3}.
\end{equation}
We denote
\begin{equation}\label{def-J}
J_{\d,u,X} = \frac1{\|f\|_1}\int_{0}^{1} \left(  \psi(e^uX(1+\d t) )-\psi(X(1+\d t) )\right) f(t) dt  .
\end{equation}
It follows from \eqref{diff-psi-theta} that
\begin{equation}\label{ineq-IJ}
I_{\d,u,X} \ge J_{\d,u,X} -\omega \sqrt{X}. 
\end{equation}
Note that \cite{RaSa} used older approximations from \cite{Sh}, which lead to $\omega =0.0325  $.
To summarize, we want to find conditions on $m,\d,u,a$ so that
\begin{equation}\label{ineq-5}
J_{\d,u,X} -\omega  \sqrt{X} -  2(1+\d )(e^u-1)\frac{\log(e^uX(1+\d ))}{\log(X(e^u-1))} \frac{ \nu(f,a) }{\|f\|_1} X >0 .
\end{equation}
We are now left with evaluating $J_{\d,u,X}$, which we shall do by relating it to the zeros of zeta through an explicit formula.
\subsection{An explicit inequality for $J_{\d,u,X}$}
\begin{lemma} \cite[Lemma 4]{RaSa} 
Let $2\le b\le c$, and let $g$ be a continuously differentiable function on $[b,c]$. We have 
\begin{multline}
 \int_b^c\psi(u)g(u)du=\int_b^cug(u)-\sum_{\varrho}\int_b^c\frac{u^{\varrho}}{\varrho}g(u)du
\\+\int_b^c\left(\log2\pi -\frac{1}{2}\log(1-u^{-2})\right)g(u)du.
\end{multline}
\end{lemma}
We apply this identity to respectively $g(t) = f\left(\d^{-1}\left(e^{-u}X^{-1}t-1\right)\right)$, 
$b=e^u X$, $c= e^u X(1+\d)$ and  $g(t) = f\left(\d^{-1}\left(X^{-1}t-1\right)\right), b=X, c= X(1+\d)$. 
It follows that
\begin{multline*}
J_{\d,u,X} = 
\frac{\left(e^u-1\right)X}{\|f\|_1}  \int_0^1 (1+\d t) f(t) dt
- \frac1{\|f\|_1}  \sum_{\varrho}  \int_0^1  \frac{\left(e^{u\varrho}-1\right)X^{\varrho}(1+\d t)^{\varrho}  f(t)}{\varrho} dt
\\ -  \frac1{2\|f\|_1} \int_0^1 \left( \log \left( 1-\left(e^uX\left(1+\d t\right)\right)^{-2}\right) - \log \left( 1-\left(X\left(1+\d t\right)\right)^{-2}\right) \right) f(t) dt .
\end{multline*}
Observe that the last term is $  \ge - \frac{u}{2X} $.
 We obtain
  \begin{multline}\label{expl-ineq0}
\frac{J_{\d,u,X}}{(e^u-1)X}
\ge \frac{\int_0^1(1+\d t)f(t)dt}{\|f\|_1}
 - \sum_{\varrho}  \left| \frac{ \left(e^{u\varrho}-1\right) }{(e^u-1)\varrho} \frac{\int_0^1  (1+\d t)^{\varrho}  f(t) dt}{\|f\|_1}\right|  X^{\Re \varrho-1} 
\\-\frac{u}{2(e^u-1)X^2}.   
 \end{multline}
We obtain some small savings by directly computing the first term whereas 
\cite[equation (13)]{RaSa} use the following bound in \eqref{expl-ineq0} instead:
\[
\frac{  \int_0^1 (1+\d  t) f(t) dt}{\|f\|_1}
\ge \frac{u}{e^u-1}.
\]
Let $s$ be a complex number.
We denote $G_{m,\d,u}(s)$ the summand
\begin{equation}\label{def-G}
 G_{m,\d,u}(s) =  \frac{ \left(e^{us}-1\right) }{(e^u-1)s} \frac{\int_0^1  (1+\d t)^{s}  f(t) dt}{\|f\|_1},
\end{equation}
and we rewrite inequality \eqref{expl-ineq0} as
\begin{equation}\label{expl-ineq}
\frac{J_{\d,u,X}}{(e^u-1)X}\ge
G_{m,\d,u}(1) - \sum_{\varrho}\left| G_{m,\d,u}(\varrho)\right| X^{\Re \varrho-1}
-\frac{u}{2(e^u-1)}X^{-2}. 
\end{equation}
Since the right term increases with $X$, we can replace $X$ with $X_0$ for $X\ge X_0$.
Note that this is also the case for the other left term for
\[ \frac{\omega}{(e^u-1)\sqrt{X}} -  2(1+\d )\frac{\log(e^uX(1+\d ))}{\log(X(e^u-1))} \frac{ \nu(f,a) }{\|f\|_1} .\]
For simplicity we denote
\begin{equation}\label{def-Sigma}
   \Sigma  =  \Sigma_{m,\d,u,X}  = \sum_{\varrho=\beta+i\gamma} \left| G_{m,\d,u}(\varrho)\right| X^{\beta-1}.
\end{equation}
The following Proposition gives a first inequality in terms of the zeros of zeta and conditions on 
$m,u,\d,a$ (and thus $\Delta$) so that $\theta(x) - \theta(x(1-\Delta^{-1})) >0$:
\begin{prop}\label{condition}
Let $m,u,\d,a, \Delta, X_0$ satisfy \eqref{cond}.
If $X\ge X_0$ and
\begin{multline}\label{ineq-6}
G_{m,\d,u}(1) 
- \Sigma_{m,\d,u,X_0} 
-\frac{u}{2(e^u-1)}X_0^{-2} 
- \frac{\omega  }{(e^u-1)} X_0^{-1/2}
\\ -   \frac{ 2\nu(f,a)(1+\d) }{\|f\|_1}  \frac{\log(e^uX_0(1+\d ))}{\log(X_0(e^u-1))}
>0 ,
\end{multline}
then there exists a prime number between $x(1-\Delta^{-1})$ and $x$.
\end{prop}
We are now going to make this Lemma more explicit by providing computable bounds for the sum over the zeros $\Sigma_{m,\d,u,X_0} $.
\subsection{Evaluating $G_{m,\d,u}$.}\label{def-admissible}
Let $f$ be an $m$-admissible function over $[0,1]$.
We recall the properties it entitles according to the definition of \cite{RaSa}:
\begin{itemize}
\item $f$ is an $m$-times differentiable function,
\item $f^{(k)}(0)=f^{(k)}(1)=0$ for $0\le k\le m-1,$
\item $f\geq0,$
\item $f$ is not identically $0$.
\end{itemize}
Let $k=0,\ldots,m$, $s=\s+i\tau$ be a complex number with $\tau>0, 0\le \s \le 1$.
We denote
\begin{equation}\label{def-F}
 F_{k,m,\delta}=\frac{\int_0^1(1+\delta t)^{1+k}|f^{(k)}(t)|dt}{\|f\|_1}.
 \end{equation}
We provide here finer estimates than \cite{RaSa} for $G_{m,\d,u}$.
Observe that 
\begin{align}
 & \label{bnd-exp-1}
\left|\frac{ e^{us}-1 }{s}\right| = \left|\int_1^u e^{xs} dx\right| 
\le 
\int_1^u e^{x\s} dx = \frac{ e^{u\s}-1 }{\s} ,
\\& \label{bnd-exp-3}
\left|\frac{ e^{us}-1 }{s}\right| \le \frac{ e^{u\s}+1 }{\tau},\\
\text{and}\ 
& \label{bnd-f-3}
\left|\int_0^1  (1+\d t)^{s}  f(t) dt\right|
\le \frac{1}{\d^k \tau^k} F_{k,m,\delta}.
\end{align}
We deduce easily bounds for $G_{m,\d,u}(s)$ by combining \eqref{bnd-exp-1} and \eqref{bnd-f-3} with respectively 
$k=0$, 
$k=1$, 
$k=m$, 
and lastly by combining \eqref{bnd-exp-3} and \eqref{bnd-f-3} with $k=m$:
\begin{align}
 & \label{bnd-G-1}
\left|  G_{m,\d,u}(s )\right| \le 
F_{0,m,\d}  \frac{e^{u\s}-1}{(e^u-1)\s},
\\
& \label{bnd-G-5}
\left|  G_{m,\d,u}(s )\right| \le 
F_{1,m,\d} \frac{e^{u\s}-1}{(e^u-1)\s\d\tau} ,\\
 &\label{bnd-G-6}
\left|  G_{m,\d,u}(s )\right| \le 
F_{m,m,\d} \frac{e^{u\s}-1}{(e^u-1)\s\d^m \tau^m} ,\\
 & \label{bnd-G-2}
\left|  G_{m,\d,u}(s )\right| \le 
F_{m,m,\d} \frac{e^{u\s}+1}{(e^u-1)\d^m \tau^{m+1}} .
\end{align}
\subsection{Zeros of the Riemann-zeta function}
We denote each zero of zeta $\varrho = \beta + i \gamma$, $N(T)$ the number of zeros in the rectangle $0<\beta<1,0 < \gamma <T$,
and $N(\s_0,T)$ the number of those in the rectangle $\s_0<\beta <1,0 < \gamma <T$.
We assume that we have the following information.
\begin{theorem}\label{thm-info-zeros}\noindent
\begin{enumerate}
 \item A numerical verification of the Riemann Hypothesis:\\
There exists $H>2$ such that if $\zeta(\b+i\g)=0$ at $0\le \b \le1$ and $0\le \g \le H$, then $\b=1/2$.
\item A direct computation of some finite sums over the first zeros:\\
Let $0<T_0<H$ and $S_0>0$ satisfy 
\begin{align}
& \label{bnd-sum-T0-1}
\sum_{\begin{substack} { 0<\gamma \le T_0 \\ \beta=1/2} \end{substack}} 1\le N_0=N(T_0), \\
\text{and} \ & \label{bnd-sum-T0}
\sum_{\begin{substack} { 0<\gamma \le T_0 \\ \beta=1/2} \end{substack}} \frac1{\gamma} \le S_0.
\end{align}
\item A zero-free region:\\
There exists $R_0>0$ constant, such that $\zeta(\s+it)$ does not vanish in the region
\begin{equation}\label{zfr}
\s \ge 1-\frac1{R_0\log |t|}\ \text{ and }\ |t| \ge2.
\end{equation}
\item An estimate for $N(T)$:\\
There exist $a_1,a_2,a_3$ positive constants such that, for all $T\ge 2$,
\begin{equation}\label{Rosser}
\begin{split}
& |N(T) - P(T) | \le  R(T), 
\\ \text{where}\ 
& P(T)  = \frac T{2\pi}\log \frac T{2\pi} - \frac T{2\pi} + \frac{7}{8} ,\ 
R(T) = a_1 \log T + a_2 \log \log T + a_3 .
\end{split}
\end{equation}
\item An upper bound for $N(\s_0,T)$:\\
Let $3/5<\s_0<1$.
Then there exist $c_1,c_2,c_3$ constants such that, for all $T\ge H$,
\begin{equation}\label{density}
 N(\s_0,T)  \le c_1  T + c_2   \log T + c_3.
 \end{equation}
\end{enumerate}
\end{theorem}
Note that \cite{RaSa} did not use any information of the type \eqref{bnd-sum-T0}, \eqref{zfr}, or \eqref{density}. 
Instead they used \eqref{bnd-sum-T0-1}, the fact that all nontrivial zeros satisfied $\b<1$, and the classical bound \eqref{Rosser} for $N(T)$ as given in \cite{Ros}[Theorem 19]. 
Our improvement will mainly come from using a new zero-density of the form of \eqref{density}. 
\subsection{Evaluating the sum over the zeros $ \Sigma_{m,\d,u,X_0}$.}
We assume Theorem \ref{thm-info-zeros}. 
We split the sum $ \Sigma_{m,\d,u,X_0}$ vertically at heights $\g=0$ (so as to use the symmetry with respect to the $x$-axis) and consider
\[ 
\tilde{G}_{m,\d,u}(\b+i\g) = |G_{m,\d,u}(\b+i\g)| + |G_{m,\d,u}(\b-i\g)|.
\] 
We then split at $\g=H$ (so as to take advantage of the fact that all zeros below this horizontal line satisfy $\b=1/2$), and again at $\g=T_0$ and $\g=T_1$ (where $T_1$ will be chosen between $T_0$ and $H$), and consider:
\begin{align}
&\label{def-S0}   \Sigma_0 = \sum_{0<\gamma \le T_0}  \tilde{G}_{m,\d,u}(1/2+i\g) X_0^{-1/2},\\
 &   \Sigma_1 = \sum_{T_0<\gamma \le T_1}  \tilde{G}_{m,\d,u}(1/2+i\g) X_0^{-1/2},\\
 \text{ and }\  &  \Sigma_2 = \sum_{T_1<\gamma \le H}   \tilde{G}_{m,\d,u}(1/2+i\g) X_0^{-1/2}.
 \end{align}
For the remaining zeros (those with $\g>H$), we make use of the symmetry with respect to the critical line, and we split at $\b=\s_0$ for some fixed $\s_0>1/2$ (we will consider $9/10\le \s_0\le 99/100$ for our computations). 
We denote
\begin{align}
 \nonumber \Sigma_3 = & 
\sum_{\begin{substack} { \gamma>H \\ \beta=1/2} \end{substack}} \tilde{G}_{m,\d,u}(1/2+i\g) X_0^{-1/2}
\\ & + \sum_{\begin{substack} { \gamma>H \\1/2 < \beta\le \s_0} \end{substack}}  \left( \tilde{G}_{m,\d,u}(\b+i\g) X_0^{\b-1} +\tilde{G}_{m,\d,u}(1-\b+i\g) X_0^{-\b} \right),\\
  \Sigma_4  = & \sum_{ \begin{substack} { \gamma>H \\ \s_0<\beta<1}\end{substack}}\left( \tilde{G}_{m,\d,u}(\b+i\g) X_0^{\b-1} +\tilde{G}_{m,\d,u}(1-\b+i\g) X_0^{-\b} \right).
\end{align}
As a conclusion, we have 
\begin{equation}\label{eq-Sigma}
 \Sigma_{m,\d,u,X_0} =   \Sigma_0 +   \Sigma_1 +   \Sigma_2 +   \Sigma_3 +   \Sigma_4.
\end{equation}
We state here some preliminary results (see \cite[equations (2.18), (2.19), (2.20), (2.21), (2.26)]{FaKa}). 
\begin{lemma} 
Let $T_0,H,R_0,\s_0$ be as in Theorem \ref{thm-info-zeros}.
Let $m\ge 2, X_0>10$, and $T_1$ between $T_0$ and $H$.
We define
\begin{align}
& \label{def-S1}
S_1(T_1)
= \left(\frac{1}{2\pi}+q(T_0)\right) \left(\log\frac{T_1}{T_0}\log\frac{\sqrt{T_1T_0}}{2\pi}\right) \frac{2R(T_0)}{T_0}, \\
& \label{def-S2}
S_2(m,T_1)=\left(\frac{1}{2\pi}+q(T_1)\right) \left(\frac{1+m\log\frac{T_1}{2\pi}}{m^2T_1^{m}}-\frac{1+m\log\frac{H}{2\pi}}{m^2H^{m}}\right)
+\frac{2R(T_1)}{T_1^{m+1}}, \\
& \label{def-S3}
 S_3(m)=\left(\frac{1}{2\pi}+q(H)\right)\left(\frac{1+m\log\frac{H}{2\pi})}{m^2H^m}\right)+\frac{2R(H)}{H^{m+1}},
 \\
& \label{def-S4}
S_4(m,\s_0)=\Big(c_1\Big(1+\frac{1}{m} \Big) +\frac{c_2\log H}{H}+ \Big(c_3+\frac{c_2}{m+1}\Big)\frac{1}{H}\Big)\frac{1}{H^m},
\\
& \label{def-S5}
S_5(X_0,m,\s_0)=\Big(c_1+\frac{c_2 \log H}{H}+\frac{c_3}{H}+\Big(c_1+\frac{c_2}{H}\Big) \frac{R_0}{2\log X_0}\frac{(\log H)^2}{(\frac{mR_0}{\log X_0})(\log H)^2-1}\Big)\frac{1}{H^m}.
\end{align}
We assume Theorem \ref{thm-info-zeros}. Then
\begin{align}
& \label{bnd-sum-T1-A}
\sum_{T_0<\gamma \le T_1}\frac{1}{\gamma} \le S_1(T_1),\\
& \label{bnd-sum-H}
\sum_{T_1<\gamma\le H}\frac{1}{\gamma^{m+1}} \le S_2(m,T_1), \\
& \label{bnd-sum-s0}
\sum_{\gamma >H}\frac{1}{\gamma^{m+1}} \le S_3(m),\\
& \label{bnd-sum-1-s0}
\sum_{\begin{substack}{\gamma >H \\ \s_0 <\beta <1}\end{substack}}\frac{1}{\gamma^{m+1}}  \le S_4(m,\s_0).
\end{align}
Moreover, if $\log X_0 <R_0m(\log H)^2$, then
\begin{equation}
\label{bnd-sum-X_0}
 \sum_{\begin{substack}{\gamma >H \\ \s_0 <\beta <1}\end{substack}}\frac{X_0^{\frac{-1}{R_0 \log \gamma}}}{\gamma^{m+1}} \le S_5(X_0,m,\s_0)X_0^{\frac{-1}{R_0 \log H}}.
\end{equation}
\end{lemma}
\begin{lemma}\label{lemma-bnd-Sigma}
Let $m,\d,X_0$ satisfy \eqref{cond}.
We assume Theorem \ref{thm-info-zeros}.
If $\log X_0 <R_0m(\log H)^2$, then
\begin{multline}\label{bnd-Sigma}
 \Sigma_{m,\d,u,X_0} \le 
 B_0(m,\delta)X_0^{-1/2}
+ B_1(m,\delta,T_1)X_0^{-1/2}
+ B_2(m,\d,T_1)X_0^{-1/2}
\\ + B_{3}(m,\d)\left(X_0^{\s_0-1}+X_0^{-\s_0}\right)
+ B_{41}(X_0,m,\d,\s_0) X_0^{-\frac{1}{R_0\log(H)}} 
\\+ B_{42}(m,\d,\s_0) X_0^{-1+\frac{1}{R_0\log H}},
\end{multline}
where the $B_i$'s are defined in \eqref{def-B0}, \eqref{def-B1}, \eqref{def-B2}, \eqref{def-B3},\eqref{def-B41}, and \eqref{def-B42}.
\end{lemma}
\begin{proof}
We investigate two ways to evaluate $\Sigma_0$ and $\Sigma_1$.
For $\Sigma_0$, we can either combine \eqref{bnd-G-5} with \eqref{bnd-sum-T0} which computes $\sum_{ 0<\gamma \le T_0 } \gamma^{-1}$, 
or \eqref{bnd-G-1} with \eqref{bnd-sum-T0-1} which computes $\sum_{ 0<\gamma \le T_0 } 1$.
We denote 
\begin{equation}
 \label{def-B0} 
B_0(m,\d)= \min( \Sigma_{01}(m,\d), \Sigma_{02}(m,\d)),
\end{equation}
with
\begin{equation}\label{def-S01-S02}
\Sigma_{01}(m,\d)=\frac{4F_{1,m,\delta} }{(e^{u/2}+1)\d}S_0 
\ \text{ and }\ 
\Sigma_{02}(m,\d)=\frac{4F_{0,m,\delta}}{(e^{u/2}+1)}N_0 .
\end{equation}
We obtain
\begin{equation}\label{bnd-S0} 
 \Sigma_0 \le B_0(m,\d)  X_0^{-1/2}.
\end{equation}
For $\Sigma_1$, we can either combine \eqref{bnd-G-5} with the bound \eqref{bnd-sum-T1-A} for $\sum_{T_0<\gamma \le T_1} \gamma^{-1}$,  
or \eqref{bnd-G-1} with the bound \eqref{Rosser} for $N(T)$ from Theorem  \ref{thm-info-zeros}. 
We denote
\begin{equation}\label{def-B1} 
B_1(m,\d,T_1) = \min( \Sigma_{11}(m,\d,T_1), \Sigma_{12}(m,\d,T_1)), 
\end{equation}
with
\begin{equation}\label{def-S11-S12}
\Sigma_{11}(m,\d)= \frac{4F_{1,m,\delta}}{(e^{u/2}+1)\d}S_1(T_1) ,
\ \text{ and }\ 
\Sigma_{12}(m,\d)= \frac{4F_{0,m,\delta}}{e^{u/2}+1}(N(T_1)-N_0) .
\end{equation}
We obtain 
\begin{equation}\label{bnd-S1} 
 \Sigma_1 \le B_1(m,\d,T_1) X_0^{-1/2}  .
\end{equation}
It follows from \eqref{bnd-G-2} and \eqref{bnd-sum-H} that
\begin{equation}\label{bnd-S2} 
  \Sigma_2  
\le B_2(m,\d,T_1)X_0^{-1/2} , 
\end{equation}
with 
\begin{equation}\label{def-B2}
B_2(m,\d,T_1)=\frac{2F_{m,m,\delta}}{(e^{u/2}-1)\delta^{m}}S_2(m,T_1).
\end{equation}
 We use \eqref{bnd-G-2} to bound $\tilde{G}$ in $  \Sigma_3$:
 \[
  \Sigma_ 3
 \le  \frac{2F_{m,m,\d}}{(e^u-1)\d^m} 
 \sum_{ \substack{\gamma>H \\ 1/2 \le \b \le \s_0}} \frac{ (e^{u\b}+1) X_0^{\b-1} + (e^{u(1-\b)}+1)X_0^{-\b}}{\gamma^{m+1}}.
 \]
Note that since $ \log X_0 > u $, then $(e^{u\b}+1)X_0^{\b-1}+(e^{u(1-\b)}+1)X_0^{-\b}$ increases with $\b\ge 1/2$. 
 Moreover, we use \eqref{bnd-sum-s0} to bound the sum $\sum_{ \substack{\gamma>H \\ \b \ge 1/2}} \gamma^{-(m+1)}$, and obtain
 \begin{equation}
 \label{bnd-S3}
 \Sigma_ 3
  \le B_{3}(m,\d,\s_0) X_0^{\s_0-1} + B_{3}(m,\d,1-\s_0) X_0^{-\s_0},
  \end{equation}
where
 \begin{equation}
 \label{def-B3}
 B_{3}(m,\d,\s)=\frac{2F_{m,m,\d}}{\d^m}\frac{e^{u\s}+1}{e^u-1} S_3(m).
 \end{equation}
For $  \Sigma_4$ we use again \eqref{bnd-G-2} to bound $\tilde{G}$ and the fact that $X_0^{\b-1}+X_0^{-\b}$ increases with $\b$. Since $\b\le 1-\frac{1}{R_0\log \g}$ and $\g > H$ we obtain
\[
  \Sigma_4
\le \frac{2(e^{u}+1)F_{m,m,\d}}{(e^u-1)\d^m} 
\Big( \sum_{\substack{\gamma>H \\ \s_0<\b<1}}\frac{X_0^{-\frac{1}{R_0\log \gamma}}}{\gamma^{m+1}}+ X_0^{-1+\frac{1}{R_0\log H}}\sum_{\substack{\gamma>H \\ \s_0<\b<1}}\frac{1}{\gamma^{m+1}} \Big).
\]
We apply \eqref{bnd-sum-1-s0} and \eqref{bnd-sum-X_0} to bound the above sums over the zeros and obtain
\begin{equation} \label{bnd-S4}
 \Sigma_4
\le B_{41}(X_0,m,\d,\s_0) X_0^{-\frac{1}{R_0\log(H)}} + B_{42}(m,\d,\s_0) X_0^{-1+\frac{1}{R_0\log H}},
\end{equation}
with
\begin{align} 
& \label{def-B41}
B_{41}(X_0,m,\d,\s_0)= \frac{2(e^{u}+1)F_{m,m,\d}}{(e^u-1)\d^m}  S_5(X_0,m,\s_0),\\ 
& \label{def-B42}
B_{42}(X_0,m,\d,\s_0)= \frac{2(e^{u}+1)F_{m,m,\d}}{(e^u-1)\d^m}  S_4(m,\s_0).  
\end{align}
\end{proof}
Note that $G_{m,\d,u}(1)=F_{0,m,\delta}$.
Finally we apply Proposition \ref{condition} and Lemma \ref{lemma-bnd-Sigma}.
\subsection{Main Theorem.}
\begin{theorem}\label{1st-main-thm}
Let $m,u,\d,a, \Delta, X_0$, and $x$ satisfy \eqref{cond}.
Let $T_0,H,R_0,\s_0$ be as in Theorem \ref{thm-info-zeros}.
We assume Theorem \ref{thm-info-zeros}. 
If $X\ge X_0$ and
\begin{multline}\label{to-solve}
 F_{0,m,\delta}
-B_0(m,\delta)X_0^{-1/2}
- B_1(m,\delta,T_1)X_0^{-1/2}
- B_2(m,\d,T_1)X_0^{-1/2}
\\- B_{3}(m,\d,\s_0) X_0^{\s_0-1} 
 - B_{3}(m,\d,1-\s_0) X_0^{-\s_0}
- B_{41}(X_0,m,\d,\s_0) X_0^{-\frac{1}{R_0\log H }} 
\\- B_{42}(m,\d,\s_0) X_0^{-1+\frac{1}{R_0\log H}}
-\frac{u}{2(e^u-1)}X_0^{-2} 
- \frac{\omega  }{(e^u-1)} X_0^{-1/2}
\\-   \frac{ 2\nu(f,a)(1+\d) }{\|f\|_1}  \frac{\log(e^uX_0(1+\d ))}{\log(X_0(e^u-1))} >0,
\end{multline}
then there exists a prime number between $x(1-\Delta^{-1})$ and $x$.
\end{theorem}
\section{Computations.}\label{Computations}
\subsection{Introducing the Smooth Weight $f$}\label{Legendre}
We choose the same weight as \cite{RaSa}, that is 
\[
f_m(t)=(4t(1-t))^m \ \text{ if }\ 0\le t\le 1, \text{ and } 0 \text{ otherwise}.
\]
We proved in \cite{FaKa} that a primitive of $f_m$ was providing a close to optimum weight to estimate $\psi(x)$. Thus we believe that the above weight should also be close to optimal to evaluate $\psi(y)-\psi(x)$ when $y$ is close to $x$.
We recall \cite[Lemma 6]{RaSa}:
\begin{align}
& \label{bnd-f01}
\|f_m\|_1 = \frac{2^{2m}(m!)^2}{(2m+1)!},\\
& \label{bnd-fm2}
 \|f_m^{(m)}\|_2 = \frac{2^{2m}m!}{\sqrt{2m+1}}.
\end{align}
We now provide estimates for $F_{k,m,\d}$ as defined in \eqref{def-F}.  
\begin{lemma}\label{Fk-Estimates}
Let $m\ge2,\d>0$, and $0<\s<1$.We define
 \begin{align*}
 & \lambda_0(m,\d)=\frac{ (2m+1)!}{2^{2m-1}(m!)^2} ,\\ 
 & \lambda_1(m,\d)=\frac{(1+\d)^{2}(2m+1)!}{2^{2m-1}(m!)^2},\\ 
 & \lambda(m,\d)=\sqrt{\frac{(1+\d)^{2m+3}-1}{\d(2m+3)}}\frac{(2m+1)!}{m!\sqrt{2m+1}}.
\end{align*}
Then 
 \begin{align}
  &\label{bnd-F0}
1 \le F_{0,m,\d}\le 1+ \d, \\
&\label{bnd-F1}
 \lambda_0(m,\d) \le F_{1,m,\d}(\s)\le \lambda_1(m,\d), \\
&\label{bnd-Fm}
F_{m,m,\d}(\s) \le \lambda(m,\d).
 \end{align}
\end{lemma}
\begin{proof}
Inequalities \eqref{bnd-F0} follow trivially from the fact $1\le (1+\d t)\le 1+\d$.

To bound $F_{1,m,\d}$, we note that 
\[
\frac{ \|f_m'\|_1}{\|f_m\|_1} \le F_{1,m,\d} \le \frac{(1+\d)^2\|f_m'\|_1}{\|f_m\|_1}.
\]
Since $f_m'(t)$ has same sign as $1-2t$, we have
\[
\|f_m'\|_1 
=\int_1^{1/2}f_m'(t)dt-\int_{1/2}^1f_m'(t)dt
=2f_m(1/2)-f_m(0)-f_m(1)
=2.
\]
This together with \eqref{bnd-f01} achieves to prove \eqref{bnd-F1}.

Lastly, for $F_{m.m,\d}$, we apply \eqref{bnd-fm2} together with the Cauchy-Schwarz inequality:
\[
F_{m,m,\d}(\s)\le\frac{\sqrt{\int_0^1 (1+\d t)^{2(m+1)} dt } \sqrt{\int_0^1|f_m^{(m)}(t)|^2dt} }{\|f_m\|_1}
=\sqrt{\frac{(1+\d)^{2m+3}-1}{\d(2m+3)}}\frac{\|f_m^{(m)}\|_2}{\|f_m\|_1}.
\]
\end{proof}
Note that while $F_{0,m,\d}$ and  $F_{1,m,\d}$ can be easily computed as integrals, it is not the case for $F_{m,m,\d}$.
The following observation helps us to compute $F_{m,m,\d}$ directly.
We recognize in the definition of $f^{(m)}_m$ the analogue of Rodrigues' formula for the shifted Legendre polynomials:
\[
f^{(m)}_m(t)=4^m m!P_m(1-2t),
\]
where $P_m(x)$ is the $m^{th}$ Legendre polynomial, and
\[
P_m(1-2t) = (-1)^m \sum_{k=0}^m {m \choose k} {m+k \choose k} (-t)^{k}.
\]
For each each $P_m(1-2t)$, we denote $r_{j,m}$, with $j=0,\ldots,m$, its $m+1$ roots. 
Since $P_m(1-2t)$ alternates sign between each of them, we have
\begin{multline*}
 F_{m,m,\d} 
= \frac{\int_0^1 (1+\delta t)^{m+1}|P_m(1-2t)|dt}{\|f\|_1}
\\= \frac{1}{\|f\|_1}\sum_{j=0}^{m-1} (-1)^j  \int_{r_j}^{r_{j+1}} (1+\delta t)^{m+1} P_m(1-2t)dt ,
\end{multline*}
and GP-Pari is able to compute quickly this sum of polynomial integrals.
\subsection{Explicit results about the zeros of the Riemann zeta function}\label{preliminary-lemmas}
We provide here the latest values for the constants appearing in Theorem \ref{thm-info-zeros}:
\begin{theorem}\label{thm-info-zeros-cts}
\noindent
\begin{enumerate}
 \item A numerical verification of the Riemann Hypothesis (Platt \cite{Pla0}):
\[
H = 3.061\cdot 10^{10}.
\]
\item A direct computation of some finite sums over the first zeros \\(using A. Odlyzko's list of zeros):
\[
\text{For}\ T_0=1\,132\,491,
\ N_0=N(T_0)=2\,001\,052,
\ \text{and}\ S_0= 11.637732363.
\]
\item A zero-free region (Kadiri \cite[Theorem 1.1]{Kad}):
\[R_0=5.69693.\]
\item An estimate for $N(T)$ (Rosser \cite[Theorem 19]{Ros}):
\[a_1 = 0.137, \ a_2 =0.443, \ a_3 = 1.588.\]
\newline
\item An upper bound for $N(\s_0,T)$ (Kadiri \cite{Kad2}):
For all $T\ge H$, 
\[
N(\s,T)\le c_1 T + c_2 \log T + c_3,
\]
where the $c_i$'s are given in Table \ref{table1}.
\begin{table}[ht]
\centering
\caption{ $N(\s,T) \le c_1  T + c_2  \log T + c_3.$} 
\label{table1}
\begin{tabular}{|r|r|r|r|r|r|}
\hline
$\s$ & $c_1 $ & $c_2 $ & $c_3 $
\\
\hline
$0.90$ & $5.8494$ & $0.4659$ & $-1.7905\cdot10^{11}$ \\ \hline
$0.91$ & $5.6991$ & $0.4539$ & $-1.7444\cdot10^{11}$ \\ \hline
$0.92$ & $5.5564$ & $0.4426$ & $-1.7007\cdot10^{11}$ \\ \hline
$0.93$ & $5.4206$ & $0.4318$ & $-1.6592\cdot10^{11}$ \\ \hline
$0.94$ & $5.2913$ & $0.4215$ & $-1.6196\cdot10^{11}$ \\ \hline
$0.95$ & $5.1680$ & $0.4116$ & $-1.5819\cdot10^{11}$ \\ \hline
$0.96$ & $5.0503$ & $0.4023$ & $-1.5458\cdot10^{11}$ \\ \hline
$0.97$ & $4.9379$ & $0.3933$ & $-1.5114\cdot10^{11}$ \\ \hline
$0.98$ & $4.8304$ & $0.3848$ & $-1.4785\cdot10^{11}$ \\ \hline
$0.99$ & $4.7274$ & $0.3766$ & $-1.4470\cdot10^{11}$ \\ \hline
\end{tabular}
\end{table}
\end{enumerate}
\end{theorem}
Note that \cite[Theorem 19]{Ros} was recently improved by T. Trudgian in \cite[Corollary 1]{Tru2} with $a_1=0.111,\ a_2=0.275,\ a_3=2.450$.
Our results are valid with either Rosser's or Trudgian's bounds.
\subsection{Understanding the contribution of the low lying zeros}\label{B01vsB41}
We assume Theorem \ref{thm-info-zeros-cts} and that
\begin{equation}\label{cond-mdt}
m\ge m_0=5, \d < \d_0=2\cdot10^{-8}, \text{ and } T_1>t_1=10^9 
\end{equation}
(this would be consistent with the values we choose in Table \ref{Table}).
We observe that 
\[
B_0(m,\d) = \Sigma_{02} \  \text{and} \ B_1(m,\d,T_1) = \Sigma_{12}.
\] 
where $\Sigma_{02}$ and $\Sigma_{12}$ are defined in \eqref{def-S01-S02} and \eqref{def-S11-S12} respectively.
In other words, it turns out that we obtain a smaller bound for the sum over the small zeros ($0<\gamma<T$) by using $N(T)$ directly instead of evaluating $\sum_{0<\gamma<T} \gamma^{-1}$.
This essentially comes from the fact that our choice of parameters insures us with $\d \ll \frac{F_{1,m,\d} S_0}{F_{0,m,\d} N_0}$ and $\d \ll \frac{F_{1,m,\d} S_1(T_1)}{F_{0,m,\d} (N(T_1)-N_0)}$.
We first prove the inequality
\begin{equation}
\frac{S_1(t)}{N(t)} \ge c_0\frac{ \log t}{ t}. 
\end{equation}
\begin{proof}
We denote
\begin{align*}
& w_1 =\frac{1}{2}\left(\frac{1}{2\pi}+q(T_0)\right)=0.0795\ldots ,\ 
 w_2 = -\log(2\pi)\left(\frac{1}{2\pi}+q(T_0)\right)=-0.2925\ldots,\\ 
 &w_3 = \left(\frac{1}{2\pi}+q(T_0)\right)\left(\frac{-\log^2(T_0)}{2}+\log(T_0)\log(2\pi)\right)+\frac{2R(T_0)}{T_0}=-11.3860\ldots,\\ 
& v_1 =\frac{1}{2\pi}=0.1591\ldots ,\ 
v_2 =\frac{-\log(2\pi)}{2\pi}-1 =-1.2925\ldots,\ 
v_3 = a_1=0.137,\ 
\\&v_4 =  a_2=0.443,\
v_5= a_3+\frac{7}{8}=2.463.
\end{align*}
and
\[
S_1(t) = w_1 (\log t)^2 + w_2 \log t + w_3,\
P(t) + R(t) = v_1 t\log t  + v_2 t + v_3 \log t + v_4\log \log t+ v_5 .
\]
We have from \eqref{def-S1} and Theorem \ref{thm-info-zeros-cts} (d) that
\[\frac{S_1(t)}{N(t)}\ge \frac{S_1(t)}{P(t)+R(t)} = \frac{ w_1 (\log t)^2 + w_2 \log t + w_3}{ v_1 t\log t  + v_2 t + v_3 \log t + v_4\log \log t+ v_5}. \]
Since $t>t_1=10^9$, we deduce the bound
\begin{equation}\label{eq1}
\frac{S_1(t)}{N(t)} 
\ge c_0\frac{ \log t }{ t}, 
\end{equation}
where
\begin{equation}\label{def-c0}
c_0 =  \frac{w_1+ \frac{w_2} {\log t_1} + \frac{w_3} {(\log t_1)^2} }{ v_1 + \frac{v_3} {t_1} +\frac{v_4 \log\log t_1} {t_1\log t_1}+\frac{v_5} {t_1\log t_1} } \ge 0.7508. 
\end{equation}
\end{proof}
We now establish that $\Sigma_{01}+  \Sigma_{11}, \Sigma_{01}+  \Sigma_{12}$, and $\Sigma_{02}+  \Sigma_{11}$ are all larger than $\Sigma_{02}+  \Sigma_{12}$.
We make use of Lemma \ref{Fk-Estimates} to provide estimates for the $F_{k,m,\d}$'s, of \eqref{eq1}, and of the assumptions \eqref{cond-mdt} on $m,\d,T_1$.
\begin{proof}
We have
\begin{multline*}
( \Sigma_{01}+  \Sigma_{11}) -(  \Sigma_{02}+  \Sigma_{12})
= \frac{4}{e^{u/2}+1}  \left( \frac{F_{1,m,\delta}}{\d} \left( S_0 +  S_1(T_1) \right) -  F_{0,m,\delta} N(T_1) \right)
\\
 > \frac{4(1+\d) N(T_1)}{e^{u/2}+1}  \left( \frac{ (2m_0+1)!}{2^{2m_0-1}(m_0!)^2 } \frac1{\d_0(1+\d_0) } \left(\frac{ S_0 }{P(t_1)+R(t_1)} + c_0 \frac{\log t_1}{t_1}  \right)- 1 \right) >0,
\end{multline*}
since the right term between brackets is $>2.4796-1>0$.
We have
\begin{multline*}
   (\Sigma_{01}+  \Sigma_{12}) - (  \Sigma_{02}+  \Sigma_{12}) 
= \left( \frac{S_0}{ \d}F_{1,m,\delta} - N_0 F_{0,m,\delta} \right) \frac{4}{e^{u/2}+1}
\\
>  \frac{4(1+\d) N_0}{e^{u/2}+1} \left( \frac{ (2m_0+1)!}{2^{2m_0-1}(m_0!)^2} \frac1{ \d_0(1+\d_0) }\frac{S_0}{N_0} - 1  \right) >0
\end{multline*}
since the right term between brackets is $>1574-1$.
Finally,
\begin{multline*}
   (\Sigma_{02}+  \Sigma_{11}) - (  \Sigma_{02}+  \Sigma_{12})
=\frac{4}{e^{u/2}+1} \left(\frac{F_{1,m,\delta}}{\d}S_1(T_1)- F_{0,m,\delta} (N(T_1)-N_0)\right)
\\
> \frac{4(1+\d)(N(T_1)-N_0)}{e^{u/2}+1} \left(\frac{ (2m_0+1)!}{2^{2m_0-1}(m_0!)^2}\frac{1}{\d_0(1+\d_0)}\frac{S_1(t_1)}{(\frac{S(t_1)t_1}{c_0\log t_1}-N_0)}-1 \right)
>0
\end{multline*}
since the right term between brackets is $>1.3737-1$.
\end{proof}
The values for $T_1$ and $a$ given in the next table are rounded down to the last digit.
\begin{table}[ht!]
\centering
\caption{
For all $x\ge x_0$, there exists a prime between $x(1-\Delta^{-1})$ and $x$.}
\label{Table}
\begin{tabular}{@{}r r r r r r r@{}}
\toprule
$\log x_0$ &  $m $ & $\d$ &  $T_1$ & $\s_0$ & $a$ & $\Delta$ 
\\
\midrule
$\log(4\cdot 10^{18})$&$5$&$3.580\cdot10^{-8}$&$272~519~712$&$0.92$&$0.2129$&$36~082~898$
\\
$43$&$5$&$3.349\cdot10^{-8}$&$291~316~980$&$0.92$&$0.2147$&$38~753~947$
\\
$44$&$6$&$2.330\cdot10^{-8}$&$488~509~984$&$0.92$&$0.2324$&$61~162~616$
\\
$45$&$7$&$1.628\cdot10^{-8}$&$797~398~875$&$0.92$&$0.2494$&$95~381~241$
\\
 $46$ & $8$   &$1.134\cdot10^{-8}$   &$1~284~120~197$   &$0.92$   & $ 0.2651$  & $148~306~019$ 
\\ 
$47$ & $9$ & $8.080\cdot10^{-9}$ & $ 1~996~029~891$ & $0.92$ & $0.2836$ & $227~619~375$
\\
$48$ &$11$&$6.000\cdot10^{-9}$&$3~204~848~430$&$0.93$&$0.3050$&$346~582~570$
\\
$49$ &$15$&$4.682\cdot10^{-9}$&$ 5~415~123~831$&$0.93$&$0.3275$&$518~958~776$
\\
$50$ &$20$&$3.889\cdot10^{-9}$&$ 8~466~793~105$&$0.93$&$0.3543$&$753~575~355$
\\
$51$ &$28$&$3.625\cdot10^{-9}$&$12~399~463~961$&$0.93$&$0.3849$&$1~037~917~449$
\\
$52$ &$39$&$3.803\cdot10^{-9}$&$16~139~006~408$&$0.93$&$0.4127$&$1~313~524~036$
\\
$53$ &$48$&$4.088\cdot10^{-9}$&$18~290~358~817$&$0.93$&$0.4301$&$1~524~171~138$
\\
$54$ &$54$&$4.311\cdot10^{-9}$&$19~412~056~863$&$0.93$&$0.4398$&$1~670~398~039$
\\
$55$ &$56$&$4.386\cdot10^{-9}$&$19~757~119~193$&$0.93$&$0.4445$&$1~770~251~249$
\\
$56$ &$59$&$4.508\cdot10^{-9}$&$20~210~075~547$&$0.93$&$0.4481$&$1~838~818~070$
\\
$57$ &$59$&$4.506\cdot10^{-9}$&$20~219~045~843$&$0.93$&$0.4496$&$1~886~389~443$
\\
$58$ &$61$&$4.590\cdot10^{-9}$&$20~495~459~359$&$0.93$&$0.4514$&$1~920~768~795$
\\
$59$ &$61$&$4.589\cdot10^{-9}$&$20~499~925~573$&$0.93$&$0.4522$&$1~946~282~821$
\\
$60$ &$61$&$4.588\cdot10^{-9}$&$20~504~393~735$&$0.93$&$0.4527$&$1~966~196~911$
\\
$150$ &$64$&$4.685\cdot10^{-9}$&$21~029~543~983$&$0.96$&$0.4641$&$2~442~159~714$
\\
\bottomrule
\end{tabular}
\end{table}

($\log(4\cdot 10^{18})=42.8328\ldots$.)
\subsection{Verification of the Ternary Goldbach conjecture}\label{GB}
\begin{proof}[Proof of Corollary \ref{Goldbach}]
Let $N=4\cdot10^{18}$. 
We follow Oliveira e Silva, Herzog and Pardi \cite{SilHerPar}'s argument where the authors computed all the prime gaps up to $4\cdot10^{18}$. 
From Table \ref{Table}, we have that for $x = e^{60}$ and $\Delta=1~966~090~061$, there exists at least one prime in the interval $(x-x/\Delta,x]$. This one has length $5.8082\cdot10^{16}$. 
Then $N\Delta=7.8647\cdot10^{27}$ and we may infer that the gap between consecutive primes up to $N\Delta$ can be no larger than $N$ (since $N\Delta/\Delta=N$). 
The corollary follows by using all the odd primes up to $N\Delta$ to extend the minimal Goldbach partitions of $4,6,\ldots,N$ up to $N\Delta$ (the method of computation is explained in \cite[Section 1]{SilHerPar}). 
We also note that $N+2=211+(N-209)$ and $N+4=313+(N-309)$, where $211,313,N-209,$ and $N-309$ are all prime. 
Thus, there is at least one way to write each odd number greater than $5$ and  smaller than $N\Delta$ as the sum of at most $3$ primes.   
\end{proof}
\end{document}